\documentclass[a4paper, 12pt, reqno]{amsart}
\usepackage{amsmath}
\usepackage{amssymb}
\usepackage{graphics}
\usepackage{mathrsfs}
\usepackage[pdflatex]{graphicx}
\usepackage[marginratio=1:1,tmargin=117pt, height=650pt]{geometry}
\usepackage{color}
\usepackage{amsthm}

\usepackage[all]{xy}

\usepackage{eurosym}

\usepackage[english]{babel}
\usepackage[T1]{fontenc}
\usepackage[latin1]{inputenc}
\usepackage{marginnote}

\usepackage{amsmath,amsfonts,amsthm,amssymb,verbatim,enumerate,quotes,graphicx}
\usepackage[flushmargin]{footmisc}
\usepackage[noadjust]{cite}
\usepackage{color}
\newcommand{\comments}[1]{}
\numberwithin{equation}{section}
\makeatletter
\def\blfootnote{\xdef\@thefnmark{}\@footnotetext}
\makeatother

\definecolor{orange}{rgb}{1,0.5,0}

\theoremstyle{plain}
\newtheorem{theorem}{Theorem}[section]

\newtheorem{lemma}[theorem]{Lemma}

\theoremstyle{definition}
\newtheorem{dfn}{Definition}[section]

\theoremstyle{remark} 

\newtheorem{remark}{Remark}[section]

\begin{document}
\title[Fatou's web]{Fatou's web}
\author{V. Evdoridou}
\address{Department of Mathematics and Statistics\\ The Open University\\ Walton Hall\\ Milton Keynes MK7 6AA\\ United Kingdom}
\email{vasiliki.evdoridou@open.ac.uk}
\date{\today}

\begin{abstract}
Let $f$ be Fatou's function, that is, $f(z)= z+1+e^{-z}$. We prove that the escaping set of $f$ has the structure of a `spider's web' and we show that this result implies that the non-escaping endpoints of the Julia set of $f$ together with infinity form a totally disconnected set. We also give a well-known transcendental entire function, due to Bergweiler, for which the escaping set is a spider's web and we point out that the same property holds for families of functions.
\end{abstract}

\maketitle
\section{Introduction}

Let $f$ be a transcendental entire function and denote by $f^n, n=0,1,...,$ the $n$th iterate of $f$. The set of points $z \in \mathbb{C}$ for which $(f^n)_{n \in \mathbb{N}}$ forms a normal family in some neighborhood of $z$ is called the \textit{Fatou set} $F(f)$ and the complement of $F(f)$ is the \textit{Julia set} $J(f)$. An introduction to the properties of these sets can be found in \cite{Berg}.

The set $$I(f)= \{z \in \mathbb{C}: f^n(z) \to \infty\;\text{as}\; n \to \infty\}$$ is called the \textit{escaping set} and was first studied for a general transcendental entire function by Eremenko in \cite{Ere}. The set of points that escape to infinity as fast as possible forms a subset of $I(f)$ known as the \textit{fast escaping set} $A(f).$ This set was introduced by Bergweiler and Hinkkanen in \cite{B-H} and is defined as follows:
$$A(f)= \{z: \;\text{there exists}\;\ell \in \mathbb{N}\;\text{such that}\; \lvert f^{n+\ell}(z) \rvert \geq M^n(R,f),\;\text{for}\;n \in \mathbb{N}\},$$
where $$M(r,f)= M(r) = \max_{\lvert z\rvert =r} \lvert f(z)\rvert, \;\;\text{for}\;\;r>0,$$ and $R>0$ is large enough to ensure that $M(r) >r$ for $r \geq R.$

Rippon and Stallard showed that, for several families of transcendental entire functions, $I(f)$ has a particular structure defined as follows (see \cite{Fast}).
\begin{dfn}
A set $E$ is an \textit{(infinite)} \textit{spider's web} if $E$ is connected and there exists a sequence $(G_n)$ of bounded simply connected domains, $n \in \mathbb{N},$ with $G_n \subset~ G_{n+1}$, for $n \in \mathbb{N},$ $\partial G_n \subset E,$ for $ n \in \mathbb{N}$ and $\cup_{n \in \mathbb{N}} G_n= \mathbb{C}.$ 
\end{dfn}
They also proved that if $I(f)$ \textit{contains} a spider's web then it \textit{is} a spider's web (see \cite[Lemma 4.5]{BEconj}). Many examples are now known of functions for which $A(f)$ is a spider's web and hence $I(f)$ is a spider's web (see, for example, \cite{M-P}, \cite{Fast},    \cite{Davesw}). Until now, there was only one known example of a function (a very complicated infinite product) for which $I(f)$ is a spider's web but $A(f)$ is not a spider's web (see \cite[Theorem~1.1]{BEconj}).

In this paper we show that this property also holds for the function $f(z)= z+1+~e^{-z}.$ This function was first studied by Fatou in \cite{Fatou} and for this reason it is sometimes called Fatou's function (see \cite{KU}). Fatou showed that $F(f)$ consists of one invariant component~$U$ which is a Baker domain, that is, $f^n \to \infty$ in $U$, and $U$ contains the right half-plane. In fact, $f$ was the first example of a function with a Baker domain. The Julia set is an uncountable union of curves in the left half-plane (a Cantor bouquet) (see Figure 1). Finally, $I(f)$ consists of all the points in the complex plane except for some of the endpoints of the curves in $J(f),$ and it is connected, since $U \subset I(f)\subset \overline{U}= \mathbb{C}.$ The following theorem shows that it is in fact a spider's web.

\begin{figure}
\centering
\includegraphics[width=.7\linewidth]{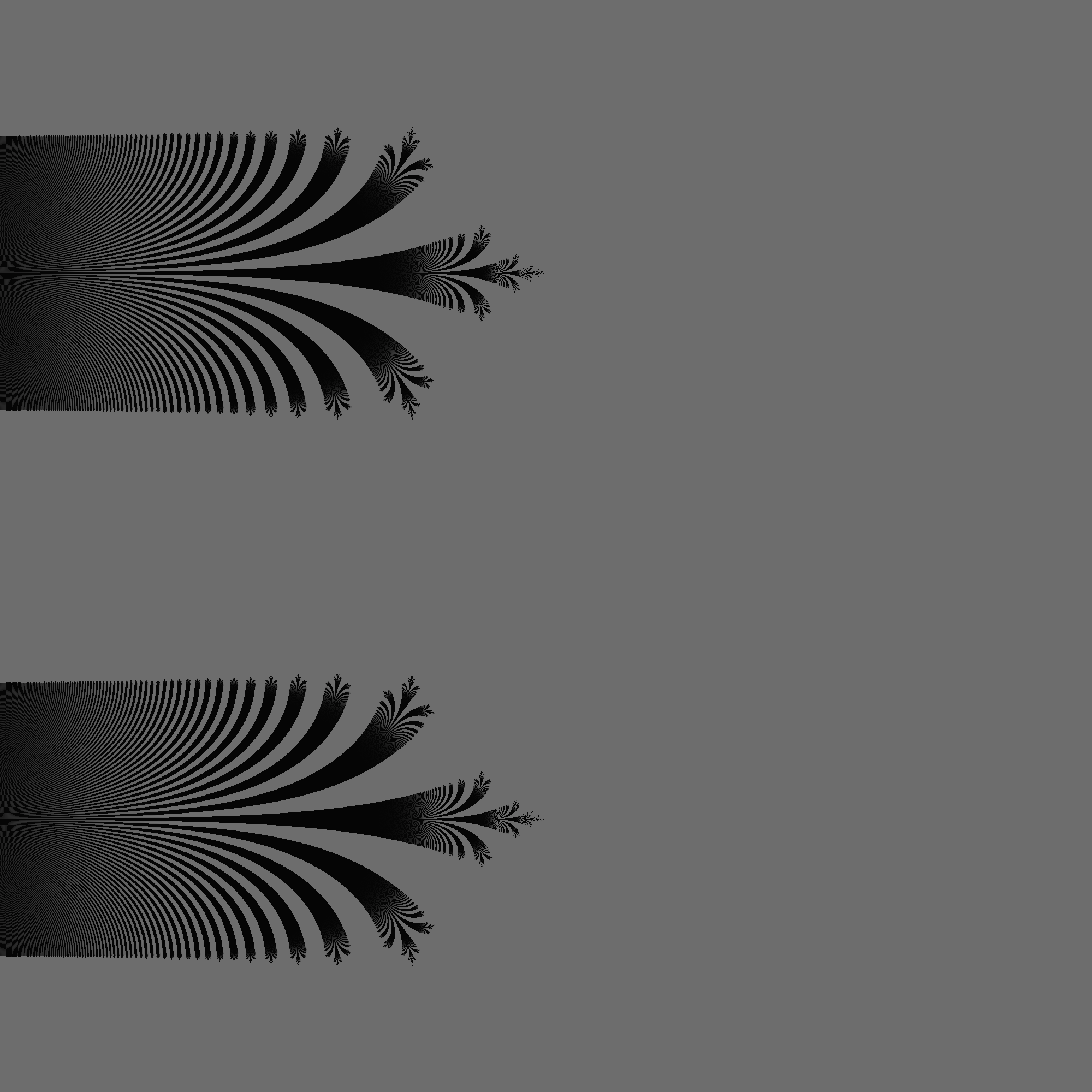}
\caption{$J(f)$ for $f(z)= z+1+e^{-z}$}
\end{figure}

\begin{theorem}
\label{Main}
Let $f(z)= z+1+ e^{-z}$. Then $I(f)$ is a spider's web.
\end{theorem}

Theorem \ref{Main} gives a positive answer to a question of Rippon and Stallard (\cite[Problem 9]{workshop}) and has some interesting consequences discussed in Section 5. Note that $A(f)$ is not a spider's web, in this case, since it consists of the curves in $J(f)$ except for some of their endpoints and so it is not connected \cite[Example 3]{Slow}.

In fact we can prove a stronger result than Theorem \ref{Main} and in order to state this we consider points which escape to infinity at a uniform rate.

Let $f$ be a transcendental entire function and $(a_n)$, $n \in \mathbb{N},$ a positive  sequence  such that $a_n \to \infty$ as $n \to \infty.$ Let $I(f,(a_n))$ be the subset of $I(f)$ defined as follows:
$$I(f, (a_n))= \{z \in \mathbb{C}: \lvert f^n(z) \rvert \geq a_n,\;\text{for}\; n \in \mathbb{N}\}.$$

Now consider the sequence
$$\left(\frac{n+m}{2}\right)= \frac{1+m}{2}, \frac{2+m}{2},...,$$
where $m \in \mathbb{N}.$ We have the following theorem.

\begin{theorem}
\label{bdcomp}
Let $f(z)= z+1+e^{-z}$. Then 
$I(f,((n+m)/2))$ contains a spider's web, for all $m \in \mathbb{N}.$

\end{theorem}
Theorem \ref{bdcomp} shows that $I(f)$ contains a spider's web and hence it is a spider's web (see \cite[Lemma 4.5]{BEconj}), and so Theorem \ref{Main} follows.

An approximation to the set $I(f,((n+m)/2))$, for $f(z)= z+1+ e^{-z}$ and $m=6$, is shown in Figure 2. The set of white points as well as light grey points of $J(f)$ is the complement of $I(f,((n+6)/2))$. Note that, since the boundary of the largest visible complementary component of $I(f,((n+6)/2))$ lies in $I(f,((n+6)/2))$, Figure 2 shows a loop in $I(f)$ that surrounds some of the non-escaping endpoints of $J(f).$

\begin{figure}
\centering
\includegraphics[width=.7\linewidth]{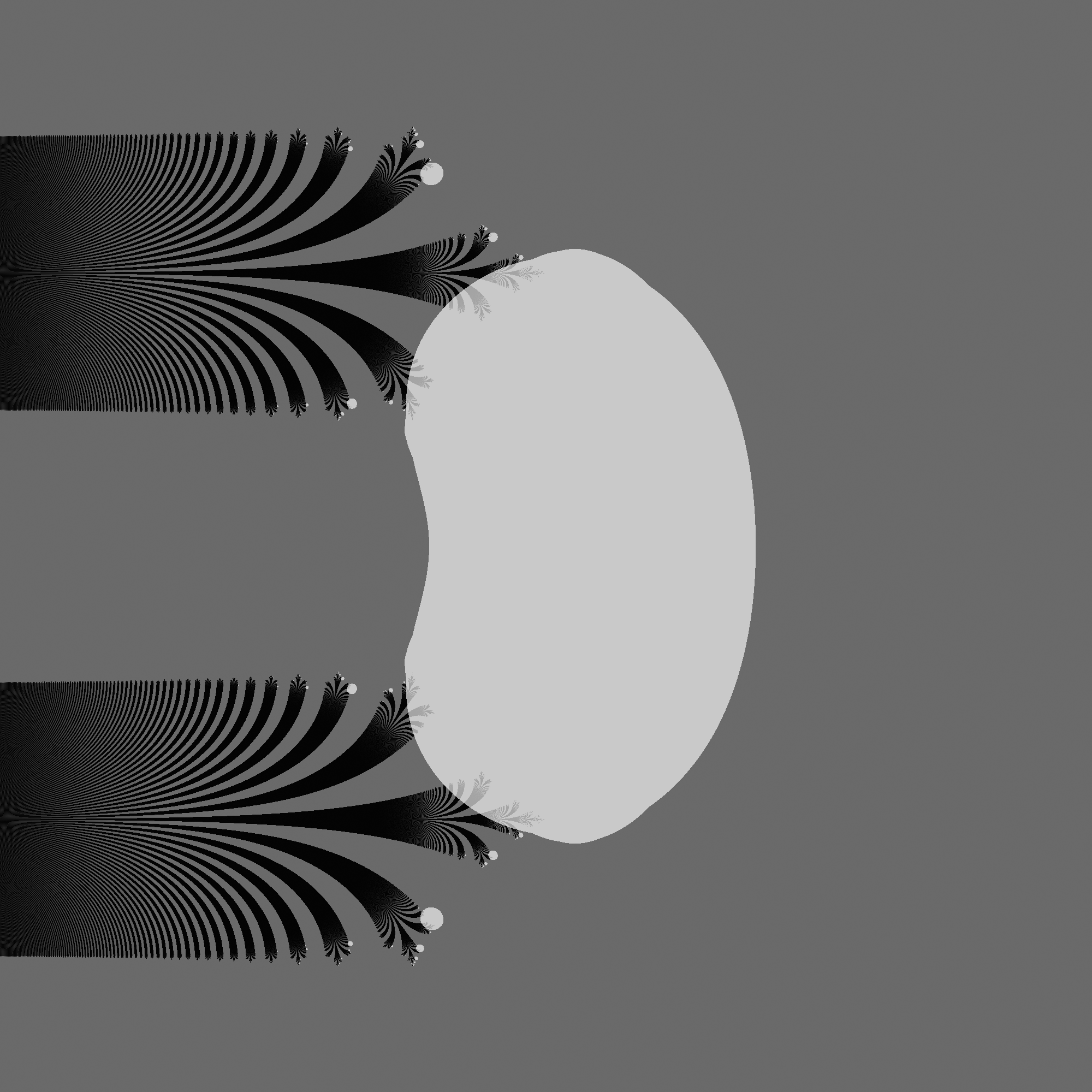}
\caption{Complementary components of $I(f,((n+6)/2))$ for $f(z)= z+1+e^{-z}$}
\end{figure}

We have the following general sufficient condition for $I(f)$ to be a spider's web, which can be used to prove Theorem \ref{Main}.

\begin{theorem}
\label{Isw}
Let $f$ be a transcendental entire function. If $I(f,(a_n))$ is defined as above,  the disc $D(0, a_n)$ contains a periodic cycle of $f$, for all $n \in \mathbb{N},$ and $I(f,(a_n))^c$ has a bounded component, then $I(f)$ is a spider's web.
\end{theorem}
In order to prove Theorem \ref{bdcomp} we need the following general result.
\begin{theorem}
\label{strongsw}
Let $f$ be a transcendental entire function. If the hypotheses of Theorem \ref{Isw} hold, $(a_n)$ is increasing and $a_{n+1} \leq M(a_n)$, for $n \in \mathbb{N},$ then $I(f,(a_n))$ contains a spider's web.
\end{theorem}

In Section 2 we give the proofs of Theorems \ref{Isw} and \ref{strongsw}. In Section 3 we prove some lemmas that we need in the proof of Theorem \ref{bdcomp}, which is given in Section~ 4.

In Section 5 we discuss an interesting consequence of Theorem \ref{Main} concerning the non-escaping endpoints of $J(f)$. The study of the endpoints and the escaping endpoints of $J(f)$ for the exponential family has been of great interest for several years (see discussion in Section 5). The fact that $I(f)$ is a spider's web for Fatou's function means that we can create loops in $I(f)$ which avoid the non-escaping endpoints of $J(f).$ Hence we are able to show that the set of non-escaping endpoints together with infinity is a totally disconnected set. We thank Lasse Rempe-Gillen for pointing out that this result is a consequence of Theorem \ref{Main}.

Finally, in Section 6 we adapt the method from the proof of Theorem \ref{bdcomp} in order to show that $I(f)$ is a spider's web for another well known function, namely, $f(z)= 2z+2- \log 2 -e^z.$ The function $f(z)= 2z+2- \log 2-e^z$ was first studied by Bergweiler in \cite{InvD} who showed that it has a Baker domain that contains a left half-plane and also has several other interesting dynamical properties. Note that our result implies that $I(f)$ is a connected set in this case, a property that was not known before.

We end the Introduction with a remark about families of functions each of which behaves in a similar way to one of the two functions we study here and hence their escaping sets are also spiders' webs.

\begin{remark}
Let $\{f: f(z)=az+b+ce^{dz}\}$ where $a,b,c,d \in \mathbb{R}$ be a family of transcendental entire functions such that:
\begin{itemize}
\item[(a)] $a=1,\; c\neq 0,\; bd<0$; or
\item[(b)] $a>1,\; c,d \neq 0.$
\end{itemize}
Then similar arguments to those used in Section 4 and Section 6 show that for the families of functions defined in (a) and (b), respectively, $I(f)$ is a spider's web.
\end{remark}

\begin{remark}
A set related to $I(f, ((n+m)/2))$ was used by Bergweiler and Peter in \cite{B-P} in relation to the Hausdorff measure for entire functions. More precisely, they considered the set 
$$\textrm{Unb}(f, (p_n))= \{z \in \mathbb{C}: \lvert f^n(z) \rvert > p_n,\;\text{for infinitely many}\; n \},$$
where $(p_n)$ is a real sequence tending to infinity.
\end{remark}
\textit{Acknowledgments.} I would like to thank my supervisors Prof. Phil Rippon and Prof. Gwyneth Stallard for their patient guidance and all their help with this paper and Prof. Lasse Rempe-Gillen for a nice observation about a consequence of Theorem \ref{Main}. I would also like to thank David Mart\'{i}-Pete for producing the pictures in the Introduction.

\section{Results on spiders' webs}
In this section we give the proofs of Theorems \ref{Isw} and \ref{strongsw}.
\begin{proof}[Proof of Theorem \ref{Isw}]
Let $G$ be a bounded component of $I(f,(a_n))^c.$ As $I(f,(a_n))$ is a closed set,   we deduce that $G$ is open and
\begin{equation}
\label{boundary}
\partial G \subset I(f,(a_n)) \subset I(f).
\end{equation}
 Now let $G_n= \widetilde{f^n(G)},$ where $\widetilde{U}$ denotes the union of $U$ and its bounded complementary components. (The notation $T(U)$ is often used to describe the set $\widetilde{U}$ --see  for example \cite{Bell}, \cite{qr}). We show that there exists a sequence $(n_k)$ such that 
\begin{equation}
\label{Gnsub} 
 G_{n_k} \subset G_{n_{k+1}}\;\text{and}\;\partial G_{n_k} \subset I(f),\;\text{for}\;k \in \mathbb{N},\;\text{and}\;\; \cup_{k} G_{n_k}= \mathbb{C}.
\end{equation} 
  First note that, for each $n \in \mathbb{N},$ $\partial G_n \subset f^n(\partial G),$ since $G$ is a bounded domain. Also, as $\partial G \subset I(f,(a_n))$, we have that 
$$f^n(\partial G) \subset \{z: \lvert z \rvert \geq a_n\}$$ and hence
\begin{equation}
\label{boundary2}
\partial G_n \subset \{z: \lvert z \rvert \geq a_n\},\;\text{for}\;n \in \mathbb{N}.
\end{equation}

In order to show that $\partial G_n$ surrounds the disc $D(0, a_n)$ for $n$ large enough, note that $G \subset I(f,(a_n))^c$ and so, for each point $z$ in $G,$ there exists $N \in \mathbb{N}$ such that $\lvert f^N(z)\rvert < a_N.$  It follows from (\ref{boundary2}) that $\partial G_N$ surrounds $D(0, a_N)$. Since the disc $D(0,a_n)$ contains a periodic cycle, for all $n \in \mathbb{N},$ it then follows from (\ref{boundary2}) that
\begin{equation}
\label{gn}
G_n \supset D(0, a_n),\;\;\text{ for all}\;\; n \geq N.
\end{equation}
Note now that, for $n \in \mathbb{N},$
\begin{equation}
\label{gn2}
\partial G_n \subset f^n(\partial G) \subset I(f),
\end{equation}
by (\ref{boundary}).
It follows from (\ref{gn}) and (\ref{gn2}), together with the fact that $a_n \to \infty$ as $n \to \infty$, that we can find a sequence $(n_k)$ such that $(G_{n_k})$ satisfies (\ref{Gnsub}).

We know that $I(f)$ has at least one unbounded component, $I_0$ say (see \cite[Theorem 1]{F-E}). Clearly $I_0$ must meet $\partial G_{n_k}$ for all sufficiently large $k$. Hence, by (\ref{Gnsub}), there exists a subset of $I(f)$ which is a spider's web and this implies that $I(f)$ is a spider's web (see \cite[Lemma 4.5]{BEconj}).
\end{proof}

\begin{proof}[Proof of Theorem \ref{strongsw}]
Let $G$ be a bounded component of $I(f,(a_n))^c$ and $G_n= \widetilde{f^n(G)}$ as before.
In addition to the fact proved in the proof of Theorem \ref{Isw}, we also need to show that $\partial G_m \subset I(f,(a_n))$, for $m \in \mathbb{N},$ and that $I(f,(a_n))$ has an unbounded component. As $(a_n)$ is increasing and $\partial G \subset I(f,(a_n))$, we have 
$$\partial G_m \subset f^m(\partial G) \subset I(f,(a_{n+m})) \subset I(f,(a_n)),\;\text{for}\;m \in \mathbb{N}.$$
Hence, the subsequence $(G_{n_k})$ found in the proof of Theorem \ref{Isw} satisfies $\partial G_{n_k} \subset I(f,(a_n))$ in addition to (\ref{Gnsub}).

Recall now that each component of 
$$A_R(f)= \{z: \lvert f^n(z) \rvert \geq M^n(R),\;\text{for}\; n \in \mathbb{N}\},$$ whenever $R>0$ is such that $M(r,f)> r$ for $r \geq R,$ is unbounded (see \cite[Theorem 1.1]{Fast}). Now assume also that $R> a_1$ and so $M(R,f) \geq a_1,$ and let $z \in A_R(f)$. Then $\lvert f^n(z) \rvert \geq M^n(R)$, for $n \in \mathbb{N}.$ As $a_{n+1} \leq M(a_n),$ for all $n \in \mathbb{N},$ we deduce that $M^n(R) \geq a_n,$ for all $n \in \mathbb{N}$, and so $z \in I(f,(a_n)).$ Therefore, $I(f,(a_n))$ has at least one unbounded component. This will meet all $\partial G_{n_k}$, for $k$ large enough, and hence $I(f,(a_n))$ contains a spider's web.
\end{proof}
\section{Preliminary lemmas}

For $m \in \mathbb{N}$ and $(a_n)_{n \in \mathbb{N}}= (\frac{n+m}{2})_{n \in \mathbb{N}},$ the definition of $I(f,(a_n))$ is
$$I(f,((n+m)/2))=\{z \in \mathbb{C}: \lvert f^n(z)\rvert \geq (n+m)/2, n \in \mathbb{N}\},\;\text{for}\;m \in \mathbb{N}.$$
In this section we prove some basic results concerning the function $f$, defined by $f(z)= z+1+e^{-z}$, and the sets $I(f,((n+m)/2))$, which we will use in the proof of Theorem \ref{bdcomp}.
We first show that each set $I(f,((n+m)/2)),$ $m \in \mathbb{N},$ contains a right half-plane as well as a family of horizontal half-lines and a family of horizontal lines.

\begin{lemma}
\label{Inm}
Let $f(z)= z+1+e^{-z}.$ For $ m \in \mathbb{N}$, $I(f,((n+m)/2))$ contains the following sets:
\begin{itemize}
\item[(1)] the right half-plane $\{ x+yi: x \geq m\}$;
\item[(2)] the half-lines of the form $\{x+2j\pi i: x \leq -m\},$ for each $ j \in \mathbb{Z}$ with $\lvert j \rvert <m$; and
\item[(3)] the lines of the form $\{x+ 2j\pi i: x \in \mathbb{R} \},$ for each $ j \in \mathbb{Z}$ with $\lvert j \rvert \geq m.$
\end{itemize}
\end{lemma}
\begin{proof}

Case (1): If $z= x+ yi$ and $\text{Re}(z)= x \geq m$, then $$\text{Re}(f(z))= x+1 + e^{-x} \cos y \geq x+ \frac{1}{2},$$
since $e^{-x} <1/2,$ for $x \geq 1,$ Thus,  $\text{Re}(f^n(z))\geq m+ n/2 > (n+m)/2,$ for $n \in \mathbb{N}$, and so $z \in I(f, ((n+m)/2)).$ \\

Case (2): Let $z= x+yi$ with $x \leq -m$ and  $y= 2k\pi i,$ where $k \in \mathbb{Z}$ and $\lvert k \rvert <m.$ Then
$$\text{Re}(f(z))= x+1+e^{-x} \geq -m+1 +e^{m}>m+1,$$ 
and so, by the argument in case (1), $z \in I(f,((n+m)/2)).$\\

Case (3):
Let  $z_0= 2m \pi i.$ Then, for $1 \leq n\leq m$, we have
$$\lvert f^n(z_0) \rvert > \text{Im}(f^n(z_0)) = 2m \pi\geq (m+n) \pi > \frac{n+m}{2}.$$

Also, for $n>m$, we have $$\lvert f^n(z_0) \rvert >\text{Re}(f^{n}(z_0) ) >n> \frac{n+ m}{2}.$$ 

Hence $z_0 \in I(f,((n+m)/2)).$ For any other point of the form $z_1=x+ 2m\pi i$, $ x \neq 0$, we have $f(z_1)=x'+ 2m\pi i,$ where $x'>2$. Hence $\text{Im}(f^n(z_1))=\text{Im}(f^n(z_0))=~ 2m\pi$. Since $f(x)= x+1+e^{-x}$ is increasing for $x>0$ and $x'> \text{Re}(f(z_0)),$ we deduce that $\text{Re}(f^{n}(z_1))> \text{Re}(f^{n}(z_0)).$  Therefore, $z \in I(f,((n+~m)/~2)),$ for $z= x+ 2m\pi i, x \in \mathbb{R}.$

 Now take $z_2= x +2k\pi i,$ where $k \in \mathbb{Z}$ with $\lvert k \rvert \geq m$ and $x \in \mathbb{R}.$ Then $$\text{Im}(f^n(z_2)) \geq 2m\pi= \text{Im}(f^n(x+2m\pi i)),\;\text{ and}\; \text{Re}(f^n(z_2))= \text{Re}(f^n(x)),$$ so we deduce that $z \in I(f((n+m)/2)),$ for all $z=x+2k\pi i, k \in \mathbb{Z}$ with $\lvert k \rvert \geq m.$
 \end{proof}

Now fix $m \in \mathbb{N}.$ Next we define the rectangles $$R_k= \{x+yi :\lvert x \rvert<m+k, \lvert y \rvert <2(m+k)\pi\},\;\;k \geq 0.$$ Note that
\begin{equation}
\label{squares}
R_k \subset \{z: \lvert z \rvert < 3\pi (m+k)\},
\end{equation}
 since, for $z \in R_k,$  $$\lvert z \rvert \leq \sqrt{(m+k)^2+ 4\pi^2 (m+k)^2}< 3\pi (m+k).$$We have the following lemma concerning the rectangles $R_k$.
\begin{lemma}
\label{strips}
Let $f(z)=z+1+e^{-z}.$ If $z \in I(f,((n+m+k)/2))^c$, where $ m \in~ \mathbb{N}, k \geq ~0$ are fixed, and $f(z)$ lies outside $R_{k+1},$ then $f(z)$ lies in a horizontal half-strip of the form  $$S_{k+1,j} = \begin{cases}  \{x+yi: x\leq -m-k-1,\; 2j\pi \leq y \leq 2(j+1)\pi\}, &\mbox{if } -m-k-1 \leq j<m+k+1, \\ 
\{x+yi: x\leq m+k+1,\; 2j\pi \leq y \leq 2(j+1)\pi\}, & \mbox{otherwise}, \end{cases}$$ where $j \in \mathbb{Z}.$ 
\end{lemma}
\begin{proof}
Note first that for $m \in \mathbb{N}$ and $k \geq 0,$ if $f(z) \in I(f,((n+m+k+1)/2))$ and $\lvert f(z) \rvert \geq (m+k+1)/2$, then $z \in I(f,((n+m+k)/2)).$ Hence,  $$\text{if}\;\; z \in I(f,((n+m+k)/2))^c\;\;\text{ and}\;\; \lvert f(z) \rvert \geq (m+k+1)/2,$$ $$\text{then}\;\; f(z) \in I(f,((n+m+k+1)/2))^c.$$ Now assume that $z \in I(f,((n+m+k)/2))^c$ and $f(z)$ lies outside $R_{k+1}$. Then $\lvert f(z) \rvert \geq (m+k+1)/2$ and hence  $f(z) \in I(f,((n+m+k+1)/2))^c.$ Therefore, we  deduce by Lemma \ref{Inm}, with $m$ replaced by $m+k+1$, that $f(z)$ lies in some half-strip of the form $S_{k+1,j}, j \in \mathbb{Z}.$
\end{proof}

The last result of this section gives some basic estimates for the size of $\lvert f \rvert$ in certain sets.
\begin{lemma}
\label{modf}
Let $f(z)=z+1+e^{-z}.$ If $-\textnormal{Re}(z) \geq \lvert z \rvert /2$ and $-\textnormal{Re}(z) \geq 3,$ then  
\begin{equation}
\label{doubleineq}
\frac{1}{2} e^{-\textnormal{Re}(z)} \leq \lvert f(z) \rvert \leq  2e^{-\textnormal{Re}(z)}.
\end{equation}
\end{lemma}
\begin{proof}
In order to prove the first inequality of (\ref{doubleineq}) note that
\begin{eqnarray}
\lvert f(z) \rvert &\geq& \lvert e^{-z} \rvert -1- \lvert z \rvert \nonumber \\
&=& e^{-\text{Re}(z)} -1- \lvert z\rvert \nonumber \\
&\geq& e^{-\text{Re}(z)} -1+ 2\text{Re}(z) \nonumber \\
&\geq& \frac{1}{2} e^{-\text{Re}(z)}, \nonumber
\end{eqnarray}
since $e^t/2 \geq 1+2t$, for $t \geq 3.$

For the second inequality of (\ref{doubleineq}) we have
\begin{eqnarray}
\lvert f(z) \rvert &\leq& \lvert e^{-z} \rvert +1+ \lvert z \rvert \nonumber \\
&=& e^{-\text{Re}(z)} +1+ \lvert z \rvert \nonumber \\
&\leq & e^{-\text{Re}(z)} +1-2\text{Re}(z) \nonumber \\
& \leq & 2e^{-\text{Re}(z)} , \nonumber
\end{eqnarray}
since $e^t \geq 1+2t,$ for $t \geq 2.$
\end{proof}

\section{Proof of Theorem \ref{bdcomp}}

Again let $m \in \mathbb{N}$ be fixed.
The idea of the proof of Theorem \ref{bdcomp} is to show first that all the components of $I(f,((n+m)/2))^c$ are bounded and then use Theorem \ref{strongsw}. We are going to assume that there exists an unbounded component of $I(f,((n+~m)/2))^c$ and obtain a contradiction. In order to obtain the contradiction we will make use of the following lemma (see \cite[Lemma 1]{Slow}).
\begin{lemma}
\label{R+S}
Let $E_n, n\geq 0,$ be a sequence of compact sets in $\mathbb{C}$ and $f : \mathbb{C} \to \mathbb{\hat{C}}$
be a continuous function such that 
 $$f(E_n) \supset E_{n+1},\;\;\text{ for}\;\; n\geq 0.$$
Then there exists $\zeta \in E_0$ such that

$$f^n(\zeta)\in E_n,\;\;\text{for}\;\; n\geq 0.$$
\end{lemma}

\begin{proof}[Proof of Theorem \ref{bdcomp}]
Suppose that there exists an unbounded component $U$ of the complement of $I(f,((n+m)/2))$. We will construct a certain bounded curve $\gamma_0 \subset U$ and by considering the images $f^k(\gamma_0)$, $k \geq 0$, we will deduce, by Lemma~ \ref{R+S}, that $\gamma_0$ contains a point of $I(f,((n+m)/2))$, giving a contradiction.

Take $$S_{0,j} = \begin{cases}  \{x+yi: x\leq -m, 2j\pi \leq y \leq 2(j+1)\pi\}, &\mbox{if } -m \leq j<m, \\ 
\{x+yi: x\leq m, 2j\pi \leq y \leq 2(j+i)\pi\}, & \mbox{otherwise}, \end{cases}$$ where $j \in \mathbb{Z}.$  As $I(f,((n+m)/2))$ is closed, $U$ is open and hence, by Lemma \ref{Inm}, there exists $j_0 \in \mathbb{Z}$ such that $U$ contains $z_{0,1}, z_{0,2} \in S_{0,j_0}$ and a curve $\gamma_0 \subset U$ joining $z_{0,1}$ to $ z_{0,2}$ such that:
\begin{itemize}
\item[(a)]$\text{Re}(z_{0,1})=a_0 < -3(m+1)$;
\item[(b)] $\text{Re}(z_{0,2})= a_0-10;$
\item[(c)]$\gamma_0 \subset \{z: a_0-10 \leq \text{Re}(z) \leq a_0\} \cap S_{0,j_0};$ 
\item[(d)]$-\text{Re}(z) \geq \lvert z\rvert /2,$ for $z \in \gamma_0.$
\end{itemize}

  \begin{figure}
  \centering
  \def\svgwidth{\linewidth}
 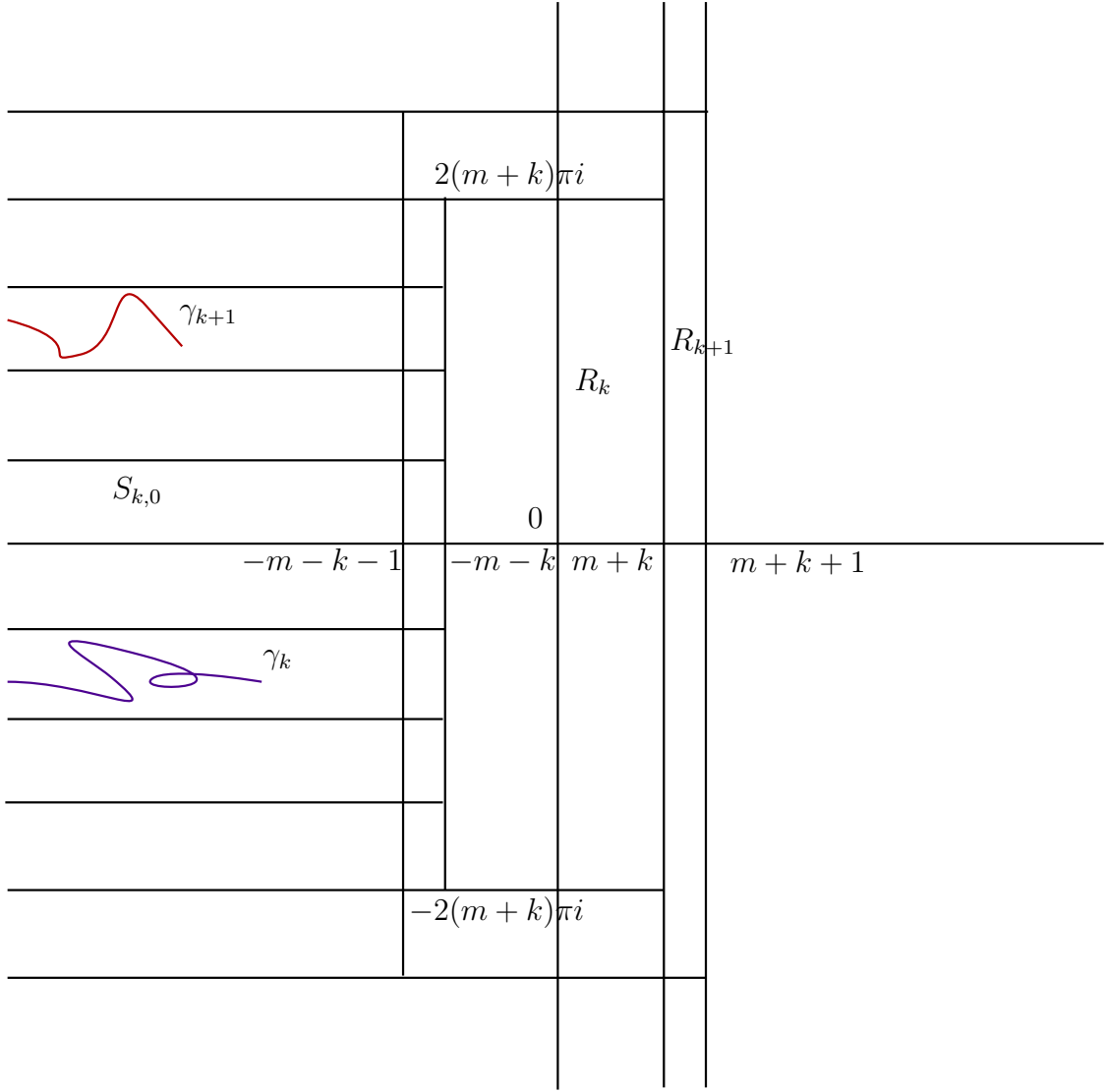
\caption{The construction of the curves $\gamma_k$ }
 \end{figure}

In general, suppose that after $k$ steps, $k \geq 0,$ we have a curve $\gamma_k \subset f(\gamma_{k-1})$ which lies in some $S_{k,j_k}$ and joins the points $z_{k,1},z_{k,2}$, such that:
\begin{itemize}
\item[(a)]$\text{Re}(z_{k,1})=a_{k}< -3(m+k+1)$;
\item[(b)] $\text{Re}(z_{k,2})= a_k-10;$
\item[(c)]$\gamma_k \subset \{z: a_k-10 \leq \text{Re}(z) \leq a_k\} \cap S_{k,j_k};$ 
\item[(d)]$-\text{Re}(z) \geq \lvert z\rvert /2$, for $z \in \gamma_k.$\\ 
 \end{itemize}
 Then, by (a), (c) and (d), and the left-hand inequality of (\ref{doubleineq}), we deduce that, for $z \in \gamma_k,$ 
\begin{equation}
\label{discineq} 
 \lvert f(z) \rvert \geq e^{-a_k}/2 > 3(m+k+1)\pi.
\end{equation} 
Hence, by (\ref{squares}), the curve $f(\gamma_{k})$ lies outside $R_{m+k+1}$ and so, by Lemma \ref{strips}, it must lie in some half-strip $S_{k+1,j_{k+1}}.$ Now, by both inequalities of (\ref{doubleineq}), we have
  
\begin{equation}
\label{4.2}
\lvert f(z_{k,2})\rvert \geq \frac{e^{-\text{Re}(z_{k,2})}}{2} = \frac{e^{10}}{2} e^{-a_k}\geq \frac{e^{10}}{4} \lvert f(z_{k,1})\rvert.
\end{equation}
 
Since  $f(\gamma_k) \subset S_{k+1,j_{k+1}} \subset \{z: \text{Re}(z) \leq m+k+1\},$ it follows from (\ref{discineq}) and (\ref{4.2}) by a routine calculation  that the set 
$$\{w \in f(\gamma_k): \lvert w \rvert \geq \frac{e^{10}}{8} \lvert f(z_{k,1})\rvert \}$$
contains points $z_{k+1,1}, z_{k+1,2}$ and a curve $\gamma_{k+1} \subset f(\gamma_k)$ (see Figure 3) joining $z_{k+1,1}$ to $z_{k+1,2}$ such that  
\begin{itemize}
\item[(a)] $\text{Re}(z_{k+1,1})= a_{k+1} <-3(m+k+2);$
\item[(b)] $\text{Re}(z_{k+1,2})= a_{k+1} -10;$
\item[(c)] $\gamma_{k+1} \subset \{z: a_{k+1}-10\leq\text{Re}(z) \leq a_{k+1}\} \cap S_{k+1,j_{k+1}};$
\item[(d)] $-\text{Re}(z)>\lvert z \rvert /2$,  for $z \in f(\gamma_{k}).$
\end{itemize}
 
As $\gamma_{k+1} \subset f(\gamma_{k}),$ for $k \geq 0$, it follows from  Lemma \ref{R+S} that there exists $\zeta \in \gamma_0$ such that $f^k(\zeta) \in \gamma_k$, for all $k \geq 0.$ Hence, by (\ref{discineq}), $$\lvert f^k(\zeta)\rvert > 3(k+m)\pi> \frac{k+m}{2},\;\text{for}\;k \geq 0,$$
 and so $\zeta \in I(f, ((n+m)/2))$, which gives us a contradiction. So all the components of $I(f,((n+m)/2))^c$ are bounded.
 
 Now take $m \geq 6.$ Since the disc $D(0, 7/2)$ contains two fixed points of $f$, at $\pm \pi i$, each disc $D(0, (n+m)/2),$ $n \in \mathbb{N},$ $m\geq 6$ contains two fixed points. Therefore, by Theorem \ref{strongsw}, we deduce that $I(f, ((n+m)/2))$ contains a spider's web, for all $m \geq 6$. Since $I(f, ((n+m')/2)) \supset I(f, ((n+m)/2))$ whenever $m'<m$, it follows that $I(f, ((n+m)/2))$ contains a spider's web for all $m \in \mathbb{N}.$\end{proof}

\begin{remark}
In the proof of Theorem \ref{bdcomp} we can actually obtain a contradiction by supposing that $I(f,((n+m)/2))^c$ has a component of sufficiently large diameter. In fact it follows from the proof that all the components of $I(f,((n+m)/2))^c$ that lie outside the rectangle $R_0$ have diameter less than $12$.
\end{remark}

\section{Endpoints}
Many people have studied the endpoints of the Julia sets of functions in the exponential family $f_a(z)= e^z +a$, especially in the case when $a \in (-\infty, -1)$, including the celebrated result by Karpi\'{n}ska that the set of endpoints in this case has Hausdorff dimension 2 (see \cite{Karpinska}). Mayer showed that, for $a$ as above, the set of endpoints, $E(f_a),$ is a totally disconnected set whereas $E(f_a) \cup \{\infty\}$ is connected (see \cite[Theorem 3]{Mayer}). Following this result, Alhabib and Rempe-Gillen proved recently that the same is true if instead of $E(f_a)$ we consider the set of the \textit{escaping} endpoints $\tilde{E}(f_a)$ whenever $a$ and $f_a$ satisfy certain conditions (see \cite[Theorem ~1.3]{ARexplode}).

Since $J(f)$ is a Cantor bouquet for Fatou's function we can consider the set of endpoints of $J(f)$ which we denote by $E(f).$ We show that Mayer's result also holds for Fatou's function.

\begin{theorem}
\label{May}
Let $f(z)= z+1+e^{-z}.$ Then $E(f)$ is totally disconnected but $E(f) \cup \{\infty\}$ is connected.
\end{theorem}

\begin{proof}
First we show that $E(f)$ is totally disconnected. Bara\'{n}ski showed in \cite[p.~52]{Baranski} that $E(g)$ is totally disconnected whenever $g$ is a transcendental entire function of finite order such that all the singularities of $g^{-1}$ are contained in a compact subset of the immediate basin of an attracting fixed point. Note that this is the case for $g(z)= e^{-1}ze^{-z}$. Indeed, $g$ is of order $1$ and $0$ is an attracting fixed point for $g.$ Also it is not hard to see that the only two singularities of $g^{-1}$, namely $0$ and $e^{-2}$, lie in the attracting basin of $0$, which is the only Fatou component. Hence, $E(g)$ is totally disconnected and, since $f$ is the lift of $g$, we deduce that $E(f)$ is totally disconnected.

In the same paper Bara\'{n}ski proved that for the same class of functions the endpoints of the hairs of the Julia set are the only accessible points from the basin of attraction (see \cite[Theorem C]{Baranski}). In particular, this result holds for $g(z)= e^{-1}ze^{-z}$ and so for its lift, $f(z)=z+1+e^{-z}$. Therefore, \cite[Theorem 2]{Mayer} implies that for the domain $U= \mathbb{C} \setminus J(f)$ the so-called `principal set' of $U$, $P(U)= E(f) \cup \{\infty\}$, is connected. 
\end{proof}

We finish this section with a consequence of Theorem \ref{Main} which seems to be the first topological result about  \textit{non}-escaping endpoints of a Cantor bouquet Julia set of a transcendental entire function. In particular, we show that, for Fatou's function, the set of non-escaping endpoints, $\hat{E}(f)= E(f) \setminus \tilde{E}(f)= E(f)\setminus I(f)$ together with infinity is a totally disconnected set. Bara\'{n}ski, Fagella, Jarque and Karpi\'{n}ska have recently showed that $\hat{E}(f)$ has full harmonic measure with respect to the Baker domain (\cite[Example 1.6]{BFJK}). Note that the set $\hat{E}(f)$ is the radial Julia set of $f$, $J_r(f),$ as defined in \cite{Lasse}. This set was introduced by Urba\'{n}ski \cite{Urb} and McMullen \cite{McM} and there are many known results regarding its dimension (see also \cite{UZ}).

\begin{theorem}
Let $f(z)= z+1+e^{-z}.$ Then $\hat{E}(f) \cup \{\infty\}$ is totally disconnected.
\end{theorem}
\begin{proof}

Suppose that there is a non-trivial component of $\hat{E}(f) \cup \{\infty\}$. Since $I(f)$ is a spider's web, any non-escaping endpoint is separated from infinity by a loop in $I(f)$, so this component must lie in $\hat{E}(f) \subset E(f)$. Hence we obtain a contradiction, since, by Theorem \ref{May}, $E(f)$ is totally disconnected.
\end{proof}
Discussions with Lasse Rempe-Gillen suggest that techniques introduced in the paper can be adapted to show that the same result holds for the exponential family. We will investigate this further in future work.

\section{Bergweiler's web}
By adapting the method used in Section 4, we can show that $I(f)$ is a spider's web for other transcendental entire functions, in particular some that have similar properties to the Fatou function. A well known function in this category is the function $f(z)= 2z +2 -\log 2 - e^z$ that was first studied by Bergweiler in \cite{InvD}.

In order to produce a similar argument to that used to prove Theorem \ref{bdcomp}, we introduce the set
$$I(f, (\sqrt{2}^{\,n}))= \{z \in \mathbb{C}: \lvert f^n(z) \rvert \geq \sqrt{2}^{\,n},\;\text{for}\; n \in \mathbb{N}\}$$ and we outline the proof of the following result.
\begin{theorem}
\label{BW}
Let $f(z)= 2z+2-\log 2-e^z.$ Then $I(f, (\sqrt{2}^{\,n}))$ contains a spider's web, and hence $I(f)$ is a spider's web.
\end{theorem}

We first state some lemmas similar to the ones given in Section 3 which we need in the proof of Theorem \ref{BW}. The proofs of the lemmas are similar to the ones in Section 3, with minor modifications, and we omit the details.

We consider the family of sets 
$$I(f, (\sqrt{2}^{\,n+m}))= \{z \in \mathbb{C}: \lvert f^n(z) \rvert \geq \sqrt{2}^{\,n+m},\;\text{for}\; n \in \mathbb{N}\},\;\;\text{for}\;\;m \geq 0.$$
\begin{lemma}
Let $f(z)= 2z+2-\log 2-e^z.$ For $m \geq 0$, $I(f,(\sqrt{2}^{\;n+m}))$ contains the following sets:
\begin{itemize}
\item[(1)] the left half-plane $\{ x+yi: x \leq -2^{m+2}\}$;
\item[(2)] the half-lines of the form $\{x+2j\pi i: x \geq 2^{m+2}\},$ for each $ j \in \mathbb{Z}$ with $\lvert j \rvert <2^m$; and
\item[(3)] the lines of the form $\{x+ 2j\pi i: x \in \mathbb{R}\},$ for each $ j \in \mathbb{Z}$ with $\lvert j \rvert \geq 2^m.$
\end{itemize}
\end{lemma}

We fix $m=0$ and we define the rectangles $$R_k= \{ z= x+yi \in \mathbb{C}:\lvert x \rvert<2^{k+2}, \lvert y \rvert <(2^{k+1})\pi\},\;\;k \geq 0.$$ Note that
\begin{equation}
\label{squares2}
R_k \subset \{z: \lvert z \rvert < 2^{k+3}\},
\end{equation}
 since, for $z \in R_k$,  $$\lvert z \rvert \leq \sqrt{4^{k+2}+ \pi^2 4^{k+1}}< 2^{k+3}.$$We have the following lemma concerning the rectangles $R_k$.
\begin{lemma}
\label{strips2}
Let $f(z)= 2z+2-\log 2 - e^z.$ If $z \in I(f,(\sqrt{2}^{\,n+k+1}))^c$, where $ k \geq 0$ is fixed, and $f(z)$ lies outside $R_{k+1},$ then $f(z)$ lies in a horizontal half-strip of the form  $$S_{k+1,j} = \begin{cases}  \{x+yi: x\geq 2^{k+3}, 2j\pi \leq y \leq 2(j+1)\pi\}, &\mbox{if } -2^{k+1} \leq j<2^{k+1}, \\ 
\{x+yi: x\geq -2^{k+3}, 2j\pi \leq y \leq 2(j+i)\pi\}, & \mbox{otherwise}, \end{cases}$$ where $j \in \mathbb{Z}.$ 
\end{lemma}
Finally, we can show that Lemma \ref{modf} holds also whenever $\text{Re}(z)>4.$ 
\begin{lemma}
\label{modf2}
Let $f(z)= 2z+2-\log 2 - e^z.$ If $\textnormal{Re}(z) \geq \lvert z \rvert /2$ and $\textnormal{Re}(z) \geq 4,$ then  
\begin{equation}
\label{doubleineq2}
\frac{1}{2} e^{\textnormal{Re}(z)} \leq \lvert f(z) \rvert \leq  2e^{\textnormal{Re}(z)}.
\end{equation}
\end{lemma}

Now we are ready to present an argument similar to the one we used in the proof of Theorem \ref{bdcomp}. 
\begin{proof}
We show that all the hypotheses of Theorem \ref{strongsw} are satisfied and then the result will follow. The disc $D(0, \sqrt{2}^{\,n})$ contains the fixed point at $\log 2$ for all $n \in \mathbb{N}.$ Hence, it suffices to show that $I(f, (\sqrt{2}^{\,n}))^c$ has no unbounded components.

Suppose that there exists an unbounded component $U$ of the complement of $I(f,(\sqrt{2}^{\,n}).$ As before, we construct a curve $\gamma_0 \subset U$ and by considering the images $f^k(\gamma_0)$, $k \geq 0$, we deduce, by Lemma \ref{R+S}, that $\gamma_0$ contains a point of $I(f,(\sqrt{2}^{\,n}))$, giving a contradiction. More precisely, we can show that, for some $k \geq 0,$ we have a curve $\gamma_k \subset f(\gamma_{k-1})$ which lies in some $S_{k,j_k}$ and joins the points $z_{k,1},z_{k,2}$ such that:
\begin{itemize}
\item[(a)]$\text{Re}(z_{k,1})=a_{k}>2^{k+2}$;
\item[(b)] $\text{Re}(z_{k,2})= a_k+10;$
\item[(c)]$\gamma_k \subset \{z: a_k \leq \text{Re}(z) \leq a_k+10\} \cap S_{k,j_k};$ 
\item[(d)]$\text{Re}(z) \geq \lvert z\rvert /2$ for any $z \in \gamma_k.$
\end{itemize}
Then, by (\ref{squares2}), $f(\gamma_{k})$ lies outside $R_{k+1}$ and so, by Lemma \ref{strips2}, it must lie in some half-strip $S_{k+1,j_{k+1}}$ and so, by using Lemma \ref{modf2}, we can construct a curve $\gamma_{k+1} \subset f(\gamma_k)$ satisfying (a)-(d) with $k$ replaced by $k+1$.
 
As $\gamma_{k+1} \subset f(\gamma_{k}),$ for $k \geq 0$, it follows from  Lemma \ref{R+S} that there exists $\zeta \in \gamma_0$ such that $f^k(\zeta) \in \gamma_{k+1}$, for all $k \geq 0.$ Hence, by (\ref{discineq}) $$\lvert f^k(\zeta)\rvert > 2^{k+4}> \sqrt{2}^{\,k},$$
 and so $\zeta \in I(f, (\sqrt{2}^{\,n}))$, which gives us a contradiction.
 
 Hence, by Theorem \ref{strongsw}, we deduce that $I(f, (\sqrt{2}^{\,n}))$ contains a spider's web and hence, by \cite[Lemma 4.5]{BEconj} , $I(f)$ is a spider's web.
\end{proof}

\end{document}